\documentclass[11pt,leqno]{article}

\usepackage[active]{srcltx}

\usepackage{kerkis}
\usepackage{amsfonts,latexsym,amsmath,amssymb,amsthm}
\usepackage{dsfont}
\usepackage{fullpage}
\newtheorem{theorem}{Theorem}[section]
\newtheorem{lemma}[theorem]{Lemma}

\newtheorem{proposition}[theorem]{Proposition}
\newtheorem{corollary}[theorem]{Corollary}

\theoremstyle{definition}
\newtheorem{definition}[theorem]{Definition}

\theoremstyle{remark}

\newtheorem*{note*}{Note}

\numberwithin{equation}{section}

\makeatletter
\newcommand{\rank}{\mathop{\operator@font rank}}
\newcommand{\conv}{\mathop{\operator@font conv}}
\newcommand{\vol}{\mathop{\operator@font vol}}
\newcommand{\onetagright}{\tagsleft@false}
\makeatother

\newcommand{\ls}{\leqslant}
\newcommand{\gr}{\geqslant}
\renewcommand{\epsilon}{\varepsilon}

\usepackage{color}

\newcommand{\abs}[1]{\left\lvert#1\right\rvert}

\newcommand{\set}[1]{\left\lbrace#1\right\rbrace}

\usepackage{hyperref}

\begin{document}
\small

\title{\bf Threshold for the expected measure of the convex hull of random points with independent coordinates}

\medskip

\author{Minas Pafis}

\date{}

\maketitle

\begin{abstract}\footnotesize Let $\mu$ be an even Borel probability measure on ${\mathbb R}$. For every $N>n$ consider $N$
independent random vectors $\vec{X}_1,\ldots ,\vec{X}_N$ in ${\mathbb R}^n$, with independent coordinates having distribution
$\mu $. We establish a sharp threshold for the product measure $\mu_n$ of the random polytope $K_N:=\operatorname{conv}\bigl\{\vec{X}_1,\ldots
,\vec{X}_N\bigr\}$ in ${\mathbb R}^n$ under the assumption that the Legendre transform $\Lambda_{\mu}^{\ast}$ of the logarithmic moment generating function of $\mu$ satisfies the condition  $$\lim\limits_{x\uparrow x^{\ast}}\dfrac{-\ln \mu ([x,\infty
))}{\Lambda_{\mu}^{\ast}(x)}=1,$$ where $x^{\ast}=\sup\set{x\in\mathbb{R}\colon \mu([x,\infty))>0}$.
An application is a sharp threshold for the case of the product measure $\nu_p^n=\nu_p^{\otimes n}$, $p\gr 1$ with density
$(2\gamma_p)^{-n}\exp(-\|x\|_p^p)$, where $\|\cdot\|_p$ is the $\ell_p^n$-norm and $\gamma_p=\Gamma(1+1/p)$.
\end{abstract}

\section{Introduction}

Let $\mu $ be an even Borel probability measure on the real line and let $X_1,\ldots ,X_n$ be independent and identically distributed random variables, defined on some probability space $(\varOmega,\mathcal{F},P),$ each with distribution $\mu$, i.e., $\mu(B):=P(X_i\in B)$ for all $i\ls n$ and all $B$ in the Borel $\sigma$-algebra $\mathcal{B}(\mathbb{R})$ of ${\mathbb R}$.
Consider the random vector $\vec{X}=(X_1,\ldots ,X_n)$ and, 
for a fixed $N$ satisfying $N>n$, consider $N$ independent copies
$\vec{X}_1,\ldots ,\vec{X}_N$ of $\vec{X}$. The distribution of $\vec{X}$ is
$\mu_n:=\mu\otimes\cdots\otimes\mu $ ($n$ times) and the distribution of
$(\vec{X}_1,\ldots ,\vec{X}_N)$ is $\mu_n^N:=\mu_n\otimes\cdots\otimes\mu_n$ ($N$ times).
Our aim is to obtain a sharp threshold for the expected $\mu_n$-measure of the random polytope
$$K_N:=\operatorname{conv}\bigl\{\vec{X}_1,\ldots,\vec{X}_N\bigr\}.$$
In order to make the notion of a sharp threshold precise, for any $n\gr 1$ and $\delta\in \left(0,\tfrac{1}{2}\right)$ we define
the upper threshold
\begin{equation}\label{eq:kappa-1}\varrho_1(\mu_n,\delta)=\sup\{\varrho_1:\sup\{{\mathbb E}_{\mu_n^N}[\mu_n (K_N)]:N\ls \exp (\varrho_1n)\}\ls \delta \} \end{equation}
and the lower threshold
\begin{equation}\label{eq:kappa-2}\varrho_2(\mu_n,\delta )=\inf\{\varrho_2:\inf\{{\mathbb E}_{\mu_n^N}[\mu_n (K_N)]:N\gr \exp (\varrho_2n)\}\gr 1-\delta \}.\end{equation}
Then, we say that $\{\mu_n\}_{n=1}^{\infty }$ exhibits a sharp threshold if
$$\varrho (\mu_n,\delta ):=\varrho_2(\mu_n,\delta )-\varrho_1(\mu_n,\delta )\longrightarrow 0$$
as $n\to\infty $, for any fixed $\delta\in \left(0,\tfrac{1}{2}\right)$.

A threshold of this form was first established in the classical work of Dyer, F\"{u}redi and McDiarmid \cite{DFM}
for the case of the uniform measure $\mu$ on $[-1,1]$. We apply the general approach that was proposed in \cite{BGP-threshold}
and obtain an affirmative answer for a general even probability measure $\mu$ on ${\mathbb R}$ that satisfies some additional assumptions,
which we briefly explain (see Section~\ref{section:background} for more details).
We assume that $\mu$ is non-degenerate, i.e. ${\rm Var}(X)>0$. Let $$x^{\ast}  =x^{\ast}  (\mu ):=\sup\set{x\in\mathbb{R}\colon
\mu([x,\infty))>0}$$ be the right endpoint of the support of $\mu $ and set $I_{\mu}=(-x^{\ast},x^{\ast})$. Note that since $\mu$ is non-degenerate and even, we have that $x^{\ast}>0$. Let
$$g(t):= {\mathbb E}\bigl(e^{tX}\bigr):=\int_\mathbb{R} e^{tx}\, d\mu(x),\qquad t\in {\mathbb R}$$ denote
the moment generating function of $X,$ and let $\Lambda_{\mu}(t):=\ln g(t)$ be its logarithmic moment
generating function. By H\"{o}lder's inequality, $\Lambda_{\mu} $ is a convex function on ${\mathbb R}$. Consider the Legendre transform
$\Lambda_{\mu}^{\ast}:I_{\mu}\to {\mathbb R}$ of $\Lambda_{\mu}$; this is the function
$$\Lambda_{\mu}^{\ast}(x):=\sup\set{tx-\Lambda_{\mu}(t)\colon t\in\mathbb{R}}.$$
One can show (see Proposition~\ref{prop:expectation-1}) that $\Lambda_{\mu}^{\ast}$ has finite moments of all orders.

We say that $\mu$ is {\it admissible} if it is non-degenerate, i.e. ${\rm Var}_{\mu}(X)>0$, and satisfies the following conditions:
\begin{enumerate}
\item[(i)] There exists $r>0$ such that ${\mathbb E}\bigl(e^{tX}\bigr)<\infty$
for all $t\in (-r,r)$; in particular, $X$ has finite moments of all orders.
\item[(ii)] One of the following holds: (1) $x^{\ast }<+\infty $ and $P(X=x^{\ast})=0$, 
or (2) $x^{\ast }=+\infty $ and $\{\Lambda_{\mu}<\infty\}={\mathbb R}$,
or (3) $x^{\ast }=+\infty $, $\{\Lambda_{\mu}<\infty\}$ is bounded and $\mu$ is log-concave.
\end{enumerate}
Finally, we say that $\mu$ satisfies {\it the $\Lambda^{\ast }$-condition} if
$$\lim\limits_{x\uparrow x^{\ast}}\dfrac{-\ln \mu ([x,\infty ))}{\Lambda_{\mu}^{\ast }(x)}=1.$$
We often express this condition in the form $-\ln \mu ([x,\infty ))\sim\Lambda_{\mu}^{\ast }(x)$ 
as $x\uparrow x^{\ast}$, where ``$a(x)\sim b(x)$ as $x\to A$" stands for ``$\lim\limits_{x\to A}\frac{a(x)}{b(x)}=1$".
With these definitions, our main result is the following.

\begin{theorem}\label{th:final}Let $\mu $ be an admissible even probability measure on ${\mathbb R}$
that satisfies the $\Lambda^{\ast}$-condition. Then, for any $\delta\in \left(0,\tfrac{1}{2}\right)$
and any $\epsilon\in (0,1)$ there exists $n_0(\mu,\delta,\epsilon)$ such that
$$\varrho_1(\mu_n,\delta)\gr (1-\epsilon){\mathbb E}_{\mu}(\Lambda_{\mu}^{\ast})\quad\hbox{and}\quad 
\varrho_2(\mu_n,\delta)\ls (1+\epsilon){\mathbb E}_{\mu}(\Lambda_{\mu}^{\ast})$$
for every $n\gr n_0(\mu,\delta,\epsilon)$. In particular, $\{\mu_n\}_{n=1}^{\infty }$ exhibits a sharp threshold, i.e.
$\lim\limits_{n\to\infty}\varrho(\mu_n ,\delta )=0$, 
with ``threshold constant" ${\mathbb E}_{\mu}(\Lambda_{\mu}^{\ast})$.
\end{theorem}

In Section~\ref{section:p} we give an application of Theorem~\ref{th:final} to the case of the product $p$-measure
$\nu_p^n:=\nu_p^{\otimes n}$. For any $p\gr 1$ we denote by $\nu_p$ the probability
distribution on ${\mathbb R}$ with density $(2\gamma_p)^{-1}\exp(-|x|^p)$, where $\gamma_p=\Gamma(1+1/p)$.
We show that $\nu_p$ satisfies the $\Lambda^{\ast}$-condition.

\begin{theorem}\label{th:p-intro}For any $p\gr 1$ we have that
\begin{equation*}\lim_{x\to\infty}\frac{-\ln (\nu_p[x,\infty))}{\Lambda_{\nu_p}^{\ast}(x)}=1.\end{equation*}
\end{theorem}

Note that the measure $\nu_p$ is admissible for all $1\ls p<\infty $; it satisfies condition (ii-3) if $p=1$
and condition (ii-2) for all $1<p<\infty$. Therefore, Theorem~\ref{th:p-intro} implies that if $K_N$ is the convex hull
of $N>n$ independent random vectors $\vec{X}_1,\ldots ,\vec{X}_N$ with
distribution $\nu_p^n$ then the expected measure ${\mathbb E}_{(\nu_p^n)^N}(\nu_p^n(K_N))$ exhibits a sharp threshold at
$N=\exp ((1\pm\epsilon){\mathbb E}_{\nu_p}(\Lambda_{\nu_p}^{\ast})n)$; for any $\delta\in \left(0,\tfrac{1}{2}\right)$
we have that $\lim\limits_{n\to\infty}\varrho(\nu_p^n ,\delta )=0$.

We close this introductory section with a brief review of the history of the problem that we study and related results.
A variant of the question, in which $\mu_n(K_N)$ is replaced by the volume of $K_N$, has been studied
in the case where $\mu $ is compactly supported. Define $$\kappa =\kappa (\mu ):=\frac{1}{2x^{\ast} }\int_{-x^{\ast}
}^{x^{\ast} }\Lambda_{\mu}^{\ast}(x) dx.$$ In \cite{Gatzouras-Giannopoulos-2009} the following threshold for
the expected volume of $K_N$ was established for a large class of compactly supported
distributions $\mu $: For every $\epsilon >0$,
\begin{equation}\label{eq:limsup}\lim_{n\rightarrow\infty }\sup\left\{ (2x^{\ast} )^{-n}
{\mathbb E}(\abs{K_N})\colon N\ls \exp ((\kappa -\epsilon)n)\right\}=0\end{equation} and
\begin{equation}\label{eq:liminf}\lim_{n\rightarrow\infty}\inf\left \{ (2x^{\ast} )^{-n} {\mathbb E}(\abs{K_N})\colon N\gr  \exp
((\kappa +\epsilon )n)\right\}=1.\end{equation}
This result generalized the work of Dyer, F\"{u}redi and McDiarmid \cite{DFM} who studied the
following two cases:
\begin{enumerate}
\item[(i)] If $\mu (\{ 1\})=\mu (\{-1\})=\tfrac{1}{2}$ then $\Lambda_{\mu} (t)=\ln (\cosh t)$ and
$\Lambda_{\mu}^{\ast}:(-1,1)\rightarrow {\mathbb R}$ is given by
$$\Lambda_{\mu}^{\ast}(x)=\tfrac{1}{2}(1+x)\ln (1+x)+\tfrac{1}{2}(1-x)\ln (1-x),$$ and the result holds with $\kappa =\ln
2-\tfrac{1}{2}$. This is the case of $\pm 1$ polytopes.
\item[(ii)] If $\mu $ is the uniform distribution
on $[-1,1]$, then $\Lambda_{\mu} (t)=\ln (\sinh t/t)$, and the result holds
with $$\kappa =\int_0^{\infty }\left(\frac{1}{u}-\frac{1}{e^u-1}\right )^2du .$$
\end{enumerate}
The generalization from \cite{Gatzouras-Giannopoulos-2009} states that if $\mu $ is an even, compactly supported, Borel probability
measure on the real line and $0<\kappa(\mu)<\infty$, then \eqref{eq:limsup} holds for every $\epsilon >0$,
and \eqref{eq:liminf} holds for every $\epsilon>0$ provided that the distribution $\mu$ satisfies the
$\Lambda^{\ast }$-condition.

Further sharp thresholds for the volume of various classes of random polytopes appear in \cite{Pivovarov-2007} and \cite{Bonnet-Chasapis-Grote-Temesvari-Turchi-2019}, \cite{Bonnet-Kabluchko-Turchi-2021} where
the same question is addressed for a number of cases where $\vec{X}_i$ have rotationally invariant densities. Exponential
in the dimension upper and lower thresholds are obtained in \cite{Frieze-Pegden-Tkocz-2020} for the case where $\vec{X}_i$
are uniformly distributed in a simplex. General upper and lower thresholds have been obtained by Chakraborti, Tkocz
and Vritsiou in \cite{Chakraborti-Tkocz-Vritsiou-2021} for some general families of distributions; see also \cite{BGP-depth}.

\section{Background and auxiliary results}\label{section:background}

As stated in the introduction, we consider an even Borel probability measure $\mu$ on the real line and a random variable $X$,
on some probability space $(\varOmega,\mathcal{F},P),$ with distribution
$\mu$. In order to avoid trivialities we assume that ${\rm Var}_{\mu}(X)>0$, and in particular that
$p_{\mu}:=\max\{P(X=x):x\in {\mathbb R}\}<1$. Recall that $\mu $ is even if $\mu (-B)=\mu (B)$ for every Borel subset $B$ of ${\mathbb R}$.

For the proof of our main result we have to make a number of additional assumptions on $\mu$. The first one is that there exists $r>0$ such that
\begin{equation}\label{1.0.5}
{\mathbb E}\bigl(e^{tX}\bigr):=\int_\mathbb{R} e^{tx}\, d\mu(x)<\infty
\end{equation}
for all $t\in (-r,r)$. This assumption ensures that $X$ has finite moments of all orders.

We define $x^{\ast} :=\sup\set{x\in\mathbb{R}\colon \mu([x,\infty))>0}$ and
$I_{\mu}:=(-x^{\ast} ,x^{\ast} )$. Note that we may have $x^{\ast}=\infty $. Our second assumption is that if
$x^{\ast}<\infty $ then
\begin{equation}\label{1.0.6}P(X=x^{\ast})=\mu (\{x^{\ast}\})=0.\end{equation}
Let $g(t):= {\mathbb E}\bigl(e^{tX}\bigr)$ for $t\in {\mathbb R}$ and $\Lambda_{\mu}(t):=\ln g (t)$. One can easily check that $\Lambda_{\mu }$ is an even convex function and $\Lambda_{\mu }(0)=0$, therefore, $\Lambda_{\mu }$ is a non-negative function. The assumption \eqref{1.0.5}
implies that the interval $J_{\mu}:=\{\Lambda_{\mu }<\infty\}$ is a non-degenerate symmetric interval, possibly the whole real line. We
define $t^{\ast }=\sup J_{\mu}$. Then, $\Lambda_{\mu }$ is $C^{\infty }$ and strictly convex on $(-t^{\ast},t^{\ast})$ (for the first
assertion see \cite[Section~1.3]{Str} or \cite[Section~2]{Gatzouras-Giannopoulos-2007}; the strict convexity 
of $\Lambda_{\mu}$ follows from the fact that
$\Lambda_{\mu}^{\prime}$ is strictly increasing on $(-t^{\ast},t^{\ast})$, as explained below).

For every $t\in (-t^{\ast},t^{\ast})$ we define the probability measure $P_t$ on $(\varOmega,\mathcal{F})$ by
$$P_t(A):= {\mathbb E}\bigl(e^{tX-\Lambda_{\mu}(t)}\mathds{1}_A\bigr),\qquad A\in\mathcal{F}.$$
Define also $\mu_t(B):=P_t(X\in B)$ for any Borel subset $B$ of $\mathbb{R}$. Since $dP_t=e^{tX-\Lambda_{\mu}(t)}dP$
and ${\mathbb E}_{\mu}(X^ke^{tX})<+\infty $ for all $k\gr 1$ and $t\in J_{\mu}$, we see that $\mu_t$ has finite moments of
all orders. Also, differentiating twice $\Lambda_{\mu}$ and taking into account the definition of $P_t$, we check that 
\begin{equation}\label{eq:second-derivative}{\mathbb E}_t(X)=\Lambda_{\mu}'(t)\quad\hbox{and}\quad
\operatorname{Var}_t(X)=\Lambda_{\mu}''(t),\end{equation}
where ${\mathbb E}_t$ and $\operatorname{Var}_t$ denote expectation and variance with respect to $P_t$.
Notice that $P_0=P$ and $\mu_0=\mu$. Since $\mu$ is non-degenerate we have that $\mu_t(\{c\})\neq 1$
for all $c\in {\mathbb R}$ and $t\in (-t^{\ast},t^{\ast})$, which implies that
$\Lambda_{\mu}^{\prime\prime}(t)>0$ for all $t\in (-t^{\ast},t^{\ast})$. It follows that $\Lambda_{\mu}^{\prime}$
is strictly increasing and since $\Lambda_{\mu}^{\prime}(0)=0$ we conclude that $\Lambda_{\mu}$ is strictly
increasing on $[0,t^{\ast})$.

Let $m :[0,x^{\ast})\to [0,\infty )$ be defined by
$$m (x)=-\ln \mu ([x,\infty )).$$
It is clear that $m$ is non-decreasing. Observe that, from Markov's inequality, for any $x\in (0,x^{\ast} )$
and any $t\gr 0$, we have ${\mathbb E}\bigl(e^{tX}\bigr)\gr e^{tx}\mu ([x,\infty ))$, and hence,
\begin{equation}\label{eq:Lm}\Lambda_{\mu }(t)\gr tx-m(x).\end{equation}

An important case where \eqref{1.0.5} is satisfied is when $\mu $ is log-concave. Recall that a
Borel measure $\mu$ on $\mathbb R$ is called log-concave if $\mu(\lambda A+(1-\lambda)B) \gr \mu(A)^{\lambda}\mu(B)^{1-\lambda}$
for all compact subsets $A$ and $B$ of ${\mathbb R}$ and any $\lambda \in (0,1)$. A function
$f:\mathbb R \rightarrow [0,\infty)$ is called log-concave if its support $\{f>0\}$ is an interval in ${\mathbb R}$ and the
restriction of $\ln{f}$ to it is concave. Any non-degenerate log-concave probability measure $\mu $ on ${\mathbb R}$ has a log-concave density
$f:=f_{{\mu }}$. Since $f$ has finite positive integral, one can check that there exist constants $A,B>0$ such that $f(x)\ls Ae^{-B|x|}$ for all $x\in {\mathbb R}$ (see \cite[Lemma~2.2.1]{BGVV-book}). In particular, $f$ has finite moments of all orders. We refer to \cite{BGVV-book} for more information on log-concave probability measures.

The next lemma describes the behavior of $\Lambda_{\mu}$ at the endpoints of $J_{\mu}$
for a log-concave probability measure with unbounded support on ${\mathbb R}$ .

\begin{lemma}\label{lem:background-1}Let $\mu $ be an even log-concave probability measure on ${\mathbb R}$ with
$$x^{\ast} =\sup\set{x\in\mathbb{R}\colon \mu([x,\infty))>0}=+\infty .$$ If $J_{\mu}$ is a bounded interval, then $J_{\mu}=(-t^{\ast},t^{\ast})$
for some $t^{\ast}>0$ and $\lim_{t\uparrow t^{\ast}}\Lambda_{\mu }(t)=+\infty $.
\end{lemma}

\begin{proof}Let $f$ denote the density of $\mu $. Since $x^{\ast}=+\infty$, we have that
${\rm supp}(\mu)={\mathbb R}$, and hence, $f$ can be written as $f=e^{-q}$, where
$q:{\mathbb R}\to {\mathbb R}$ is an even convex function. By symmetry, it is enough to consider
the convergence of $\Lambda_{\mu }(t)$ for $t>0$.

Note that, since $q$ is even and convex on ${\mathbb R}$, we have $\lim_{x\to +\infty }q(x)=+\infty $ and the function $u(x)=\frac{q(x)-q(0)}{x}$ is increasing on $(0,\infty )$. First we observe that we cannot have $\lim_{x\to\infty }u(x)=\infty $. If this was the case then we would have $\lim_{x\to\infty }\frac{q(x)}{x}=\infty $, and hence
$$\int_0^{\infty }e^{tx}f(x)dx=\int_0^{\infty }e^{x\left(t-\frac{q(x)}{x}\right)}dx<\infty $$
for all $t>0$, i.e. $\Lambda_{\mu}(t)<\infty $ for all $t>0$, which is not our case.

Therefore, since $u$ is increasing, there exists $t^{\ast}>0$ such that
$$\lim_{x\to\infty }u(x)=\lim_{x\to\infty}\frac{q(x)-q(0)}{x}=t^{\ast}.$$
Assume that $0<t<t^{\ast}$. If $\epsilon >0$ satisfies $t+\epsilon <t^{\ast}$ then there exists $M>0$ such that $u(x)-t>\epsilon $
for all $x\gr M$ and then
$$\int_0^{\infty }e^{tx}f(x)dx=e^{-q(0)}\int_0^{\infty }e^{-x(u(x)-t)}dx<\infty ,$$
which shows that $t\in J_{\mu}$, and hence $(-t^{\ast},t^{\ast})\subseteq J_{\mu}$.

On the other hand, if $t=t^{\ast}$ then using the fact that $u(x)\ls t^{\ast}$ for all $x>0$ we get
$$\int_0^{\infty }e^{t^{\ast}x}f(x)dx=e^{-q(0)}\int_0^{\infty }e^{x(t^{\ast}-u(x))}dx=+\infty .$$
This shows that $J_{\mu}=(-t^{\ast},t^{\ast})$.

Finally, if we consider a strictly increasing sequence $t_n\to t^{\ast}$ then by the monotone convergence
theorem we get
$$e^{\Lambda_{\mu }(t_n)}=\int_0^{\infty }e^{t_nx}f(x)dx\longrightarrow \int_0^{\infty }e^{t^{\ast}x}f(x)dx=+\infty ,$$
which shows that $\lim_{t\uparrow t^{\ast}}\Lambda_{\mu }(t)=+\infty $.
\end{proof}

\begin{definition}\label{def:admissible}\rm Let $\mu $ be an even probability measure on ${\mathbb R}$. We will call
$\mu$ {\it admissible} if it satisfies \eqref{1.0.5} and \eqref{1.0.6}, as well as one of the following conditions:
\begin{enumerate}
\item[{\rm (i)}] $\mu $ is compactly supported, i.e. $x^{\ast }<+\infty $.
\item[{\rm (ii)}] $x^{\ast }=+\infty $ and $\{\Lambda_{\mu}<\infty\}={\mathbb R}$.
\item[{\rm (iii)}] $x^{\ast }=+\infty $, $\{\Lambda_{\mu}<\infty\}$ is bounded and $\mu$ is log-concave.
\end{enumerate}
Note that if $x^{\ast}<+\infty$ then $\{\Lambda_{\mu}<\infty\}={\mathbb R}$. Taking also into account
Lemma~\ref{lem:background-1} we see that, in all the cases that we consider, the interval $J_{\mu}=\{\Lambda_{\mu}<\infty\}$
is open, i.e. $J_{\mu}=(-t^{\ast},t^{\ast})$ where $t^{\ast}=\sup\,J_{\mu}$.
\end{definition}

The next lemma describes the behavior of $\Lambda_{\mu}^{\prime}$ for an admissible measure $\mu $.
The first case was treated in \cite{Gatzouras-Giannopoulos-2009}.

\begin{lemma}\label{lem:background-2}
Let $\mu $ be an admissible even Borel probability measure on the real line.
Then, $\Lambda_{\mu}^{\prime }:J_{\mu}\to I_{\mu}$ is strictly increasing and surjective. In particular,
$$\lim_{t\to\pm t^{\ast}}\Lambda_{\mu}'(t)=\pm x^{\ast} .$$
\end{lemma}

\begin{proof}We have already explained that, since $(\Lambda_{\mu}')'(t)=\Lambda_{\mu}''(t)=\operatorname{Var}_t(X)> 0$, the function
$\Lambda_{\mu}'$ is strictly increasing. Now, we consider the three cases of the lemma separately.

\smallskip

(i) From the inequality $-x^{\ast} e^{tX}\ls  X e^{tX}\ls x^{\ast} e^{tX},$ which holds with
probability $1$ for each fixed $t,$ and the formula $\Lambda_{\mu}'(t)={\mathbb E}\bigl(X e^{tX}\bigr)/{\mathbb E}\bigl(e^{tX}\bigr)$, we easily check
that $\Lambda_{\mu}'(t)\in (-x^{\ast} ,x^{\ast}  )$ for every $t\in {\mathbb R}$.

It remains to show that $\Lambda_{\mu}'$ is onto $I_{\mu}$. Let $x\in (0,x^{\ast}
)$ and $y\in (x,x^{\ast})$. Since $\Lambda_{\mu}(t)\gr ty-m(y)$ for all $t\gr 0$,
we have that $\Lambda_{\mu}\bigl( m(y)/(y-x)\bigr) \gr xm(y)/(y-x)$. It follows that if we consider
the function $q_x(t):=tx-\Lambda_{\mu}(t)$, then $q_x(0)=0$ and $q_x\bigl(
m(y)/(y-x)\bigr)\ls 0$. Since $q_x$ is concave and $q_x'(0)=x>0$,
this shows that $q_x$ attains its maximum at some point in the
open interval $\bigl(0,m(y)/(y-x)\bigr)$, and hence, $\Lambda_{\mu}'(t)=x$
for some $t$ in this interval. The same argument applies for all
$x\in (-x^{\ast} ,0)$.  Finally, for $x=0$ we have that
$\Lambda_{\mu}'(0)=x$.

\smallskip

(ii) We apply the same argument as in (i).

\smallskip

(iii) Assume that $\Lambda_{\mu}^{\prime}$ is bounded from above. Then, there exists $x>0$ such that $\Lambda_{\mu}^{\prime}(t)< x$
for all $t\in J_{\mu}$. We consider the function $q_x:J_{\mu}\to {\mathbb R}$ with $q_x(t)=tx-\Lambda_{\mu}(t)$. Then, $q_x$ is strictly
increasing. However, $\lim_{t\uparrow t^{\ast }}q_x(t)=-\infty $ because $\lim_{t\uparrow t^{\ast }}\Lambda_{\mu}(t)=+\infty $
by Lemma~\ref{lem:background-1}, which leads to a contradiction.\end{proof}

Let $\mu $ be an admissible even Borel probability measure on the real line. Lemma~\ref{lem:background-2} allows
us to define $h\colon I_{\mu}\to J_{\mu}$ by $h:=(\Lambda_{\mu}')^{-1}$. Observe that $h$ is a strictly increasing $C^\infty$ function and
\begin{equation}\label{eq:h-prime}h'(x)=\frac{1}{\Lambda_{\mu}''(h(x))}.\end{equation}
Next, consider the Legendre transform of $\Lambda_{\mu}$. This is the function
$$\Lambda_{\mu}^{\ast }(x):=\sup\set{tx-\Lambda_{\mu}(t)\colon t\in\mathbb{R}},\qquad
x\in\mathbb{R}.$$
In fact, since $tx-\Lambda_{\mu}(t)<0$ for $t<0$ when $x\in [0,x^{\ast})$, we have that
$\Lambda_{\mu}^{\ast }(x)=\sup\{tx-\Lambda_{\mu} (t)\colon t\gr 0\}$ in this case, and similarly
$\Lambda_{\mu}^{\ast }(x):=\sup\{tx-\Lambda_{\mu} (t)\colon t\ls 0\}$ when $x\in (-x^{\ast},0]$.

The basic properties of $\Lambda_{\mu}^{\ast }$ are described in the next lemma
(for a proof, see e.g. \cite[Proposition~2.12]{Gatzouras-Giannopoulos-2007}).

\begin{lemma}\label{lem:2.5}Let $\mu $ be an admissible even probability measure on ${\mathbb R}$. Then,
\begin{enumerate}
\item[{\rm (i)}]  $\Lambda_{\mu}^{\ast }\gr 0$, $\Lambda_{\mu}^{\ast }(0)=0$ and
$\Lambda_{\mu}^{\ast }(x)=\infty$ for $x\in\mathbb{R}\setminus [-x^{\ast} ,x^{\ast}
]$.
\item[{\rm (ii)}] For every $x\in I_{\mu}$ we have
$\Lambda_{\mu}^{\ast }(x)=tx-\Lambda_{\mu}(t)$ if and only if $\Lambda_{\mu}'(t)=x;$ hence
$$\Lambda_{\mu}^{\ast }(x)=x h(x)-\Lambda_{\mu}(h(x)) \qquad\text{for}\ x\in I_{\mu} .$$
\item[{\rm (iii)}] $\Lambda_{\mu}^{\ast }$ is a strictly convex
$C^\infty$ function on $I_{\mu},$ and $$(\Lambda_{\mu}^{\ast })'(x)=h(x).$$
\item[{\rm (iv)}] $\Lambda_{\mu}^{\ast }$ attains its unique minimum on
$I_{\mu}$ at $x=0$.
\item[{\rm (v)}] $\Lambda_{\mu }^{\ast }(x)\ls m(x)$ for all $x\in [0,x^{\ast})$; this is a direct consequence of \eqref{eq:Lm}.
\end{enumerate}
\end{lemma}

\begin{corollary}\label{cor:background-1}We have that $\lim_{x\uparrow x^{\ast}}\Lambda_{\mu}^{\ast }(x)=+\infty $.
\end{corollary}

\begin{proof}If $x^{\ast}=+\infty $ then the convexity of $\Lambda_{\mu}^{\ast}$ and the fact that
$(\Lambda_{\mu}^{\ast })^{\prime}(x)>0$ for all $x>0$ (which is a consequence of Lemma~\ref{lem:2.5}\,(iv)
and of the fact that $(\Lambda_{\mu}^{\ast})^{\prime\prime}=h^{\prime}>0$)
imply that $\lim_{x\uparrow x^{\ast }}\Lambda_{\mu}^{\ast }(x)=+\infty $.

Next, assume that $x^{\ast}<+\infty $. Since $\Lambda_{\mu}^\prime(t)\ls x^{\ast} $ for all
$t$, the function $t\mapsto tx^{\ast} -\Lambda_{\mu}(t)$ is non-decreasing.
Therefore, $$\Lambda_{\mu}^{\ast }(x^{\ast} )=\sup\limits_{t\in {\mathbb R}}
[tx^{\ast} -\Lambda_{\mu}(t)]= \lim\limits_{t\uparrow\infty} [tx^{\ast}
-\Lambda_{\mu}(t)].$$ However, $$\lim\limits_{t\uparrow\infty}
e^{-(tx^{\ast}-\Lambda_{\mu}(t))}=\lim\limits_{t\uparrow\infty}
e^{-tx^{\ast} }g(t) = \lim\limits_{t\uparrow\infty}
{\mathbb E}\bigl(e^{t(X-x^{\ast} )}\bigr) =
{\mathbb E}\!\left(\lim\limits_{t\uparrow\infty}e^{t(X-x^{\ast} )}\right) =
P(X=x^{\ast} ),$$ the third equality being a
consequence of the dominated convergence theorem.  It follows that
$\Lambda_{\mu}^{\ast }(x^{\ast} )=-\ln P(X=x^{\ast})=+\infty $. Since $\Lambda_{\mu}^{\ast}$
is lower semi-continuous on ${\mathbb R}$ as the pointwise supremum of the continuous functions
$x\mapsto tx-\Lambda_{\mu}(t)$, $t\in {\mathbb R}$, it follows that $\lim_{x\uparrow x^{\ast }}\Lambda_{\mu}^{\ast }(x)=+\infty $.\end{proof}

The next result generalizes an observation from \cite{BGP-threshold} which states that $\Lambda_{\mu}^{\ast }$ has finite moments
of all orders in the case where $\mu $ is absolutely continuous with respect to Lebesgue measure. The more general statement of the
next proposition can be found as an exercise in \cite{Deuschel-Stroock-book}.

\begin{proposition}\label{prop:expectation-1}Let $\mu $ be an even probability measure on ${\mathbb R}$. Then,
$$\int_{I_{\mu}}e^{\Lambda_{\mu}^{\ast }(x)/2}d\mu(x)\ls 4.$$
In particular, for all $p\gr 1$ we have that $\int_{I_{\mu}}(\Lambda_{\mu}^{\ast }(x))^p\,d\mu(x)<+\infty $.
\end{proposition}

\begin{proof}[Sketch of the proof] We define $F(x)=\mu((-\infty ,x])$ and for any fixed $z>0$ we set $\alpha(x)=F(x)-F(z)$
and $\beta(x)=\exp (I(x)/2)$ where $I(x)=0$ if $x\ls 0$ and $I(x)=\Lambda_{\mu}^{\ast}(x)$ if $x>0$. Note that $\alpha $ is right
continuous and increasing, and $\beta $ is increasing. Applying \cite[Theorem~21.67\,(iv)]{Hewitt-Stromberg-book} we write
$$\int_0^z\beta(x)d\alpha(x)+\int_0^z\alpha(x-)d\beta(x)=\alpha(z)e^{I(z+)/2}-\alpha(0-)e^{I(0-)/2},$$
where, for a function $f$, we denote $f(x+)=\lim\limits_{y\to x^+}f(y)$ and $f(x-)=\lim\limits_{y\to x^-}f(y)$. 
It follows that, for every $0<z<x^{\ast}$,
\begin{align*}\int_0^ze^{\Lambda_{\mu}^{\ast}(x)/2}d\mu(x) &= \int_0^z\beta(x)d\alpha(x)
=-\int_0^z\alpha(x-)d\beta(x)+\alpha(z)e^{I(z+)/2}-\alpha(0-)e^{I(0-)/2}\\
&\ls \int_0^ze^{-I(x)}d\beta(x)+1,
\end{align*}
where we have used the fact that $-\alpha(x-)=\mu([x,z])\ls e^{-\Lambda_{\mu}^{\ast}(x)}$ and $I(0-)=0$, $-\alpha(0-)\ls 1$.
Finally, we note that
$$\int_0^ze^{-I(x)}d\beta(x)+1=\int_0^z\beta(x)^{-2}d\beta(x)+1\ls \int_1^{\infty}t^{-2}dt+1=2,$$
because $\beta$ is strictly increasing and continuous on $[0,z]$ and $\beta(0)=1$. The result follows by symmetry.\end{proof}

We close this section by recalling the $\Lambda^{\ast}$-condition that was already mentioned in the introduction.

\begin{definition}\rm Let $\mu $ be an admissible even Borel probability measure on the real line. Recall that
$\Lambda_{\mu }^{\ast }(x)\ls m(x)$ for all $x\in [0,x^{\ast})$. We shall say that $\mu$ {\it satisfies the
$\Lambda^{\ast }$-condition} if
$$\lim\limits_{x\uparrow x^{\ast}}\dfrac{m(x)}{\Lambda_{\mu}^{\ast }(x)}=1.$$
\end{definition}

\section{Proof of the main theorem}\label{section:main}

Let $\mu $ be an admissible even Borel probability measure on the real line. Recall that $\mu_n=\mu\otimes\cdots\otimes\mu$ ($n$ times),
and hence the support of $\mu_n$ is $I_{\mu_n}=I_{\mu}^n$. The logarithmic Laplace transform of $\mu_n$ is defined by
\begin{equation*}\Lambda_{\mu_n}(\xi )=\ln\Big(\int_{{\mathbb R}^n}e^{\langle\xi
,z\rangle }d\mu_n (z)\Big),\qquad \xi\in {\mathbb R}^n\end{equation*}
and the Cramer transform of $\mu_n$ is the Legendre transform of $\Lambda_{\mu_n}$, defined by
\begin{equation*}\Lambda_{\mu_n}^{\ast }(x)= \sup_{\xi\in {\mathbb R}^n} \left\{ \langle x, \xi\rangle -\Lambda_{\mu_n}(\xi )\right\},
\qquad x\in {\mathbb R}^n.\end{equation*} Since $\mu_n$ is a product measure, we can easily check that
$\Lambda_{\mu_n}^{\ast }(x)=\sum_{i=1}^n\Lambda_{\mu}^{\ast }(x_i)$ for all $x=(x_1,\ldots ,x_n)\in I_{\mu_n}$, which implies that
$$\int_{I_{\mu_n}}e^{\Lambda_{\mu_n}^{\ast }(x)/2}d\mu_n(x)
=\prod_{i=1}^n\left(\int_{I_{\mu}}e^{\Lambda_{\mu}^{\ast }(x_i)/2}d\mu(x_i)\right)<+\infty .$$
In particular, for all $p\gr 1$ we have that $\int_{I_{\mu_n}}(\Lambda_{\mu_n}^{\ast }(x))^p\,d\mu_n(x)<+\infty $.
We also define the parameter
\begin{equation}\label{eq:beta}\beta(\mu_n)=\frac{{\rm Var}_{\mu_n }(\Lambda_{\mu_n}^{\ast })}{({\mathbb E}_{\mu_n }(\Lambda_{\mu_n }^{\ast }))^2}.\end{equation}
Since $\mu_n=\mu\otimes\cdots\otimes\mu$, we have ${\rm Var}_{\mu_n}(\Lambda_{\mu_n}^{\ast})=n{\rm Var}_{\mu}(\Lambda_{\mu}^{\ast})$
and ${\mathbb E}_{\mu_n}(\Lambda_{\mu}^{\ast })=n{\mathbb E}_{\mu}(\Lambda_{\mu}^{\ast })$.  Therefore,
$$\beta(\mu_n)=\frac{{\rm Var}_{\mu_n}(\Lambda_{\mu_n}^{\ast})}{({\mathbb E}_{\mu_n }(\Lambda_{\mu_n}^{\ast }))^2}
=\frac{\beta(\mu)}{n},$$
where $\beta(\mu)$ is a finite positive constant which is independent of $n$. In particular,
$\beta(\mu_n)\to 0$ as $n\to\infty$.

In order to estimate $\varrho_i(\mu_n,\delta)$, $i=1,2$, we shall follow the approach of
\cite{BGP-threshold}. For every $r>0$ we define
\begin{equation*}B_r(\mu_n):=\{x\in{\mathbb R}^n:\Lambda^{\ast}_{\mu_n}(x)\ls r\}.\end{equation*}
Note that, since $\Lambda_{\mu_n}^{\ast}(x)=\sum_{i=1}^n\Lambda_{\mu}^{\ast}(x_i)$ for all $x=(x_1,\ldots ,x_n)$
and $\Lambda_{\mu}^{\ast}(y)$ increases to $+\infty $ as $y\uparrow x^{\ast}$, for every $r>0$ there exists $0<M_r<x^{\ast}$
such that $B_r(\mu_n)\subseteq [-M_r,M_r]^n\subseteq I_{\mu}^n$, and hence $B_r(\mu_n)$ is a compact subset of $I_{\mu}^n$.

For any $x\in {\mathbb R}^n$ we denote by ${\cal H}(x)$ the set
of all half-spaces $H$ of ${\mathbb R}^n$ containing $x$. Then we define
$$\varphi_{\mu_n} (x)=\inf\{\mu_n (H):H\in {\cal H}(x)\}.$$
The function $\varphi_{\mu_n}$ is called Tukey's half-space depth. We refer the reader to the survey article of Nagy, Sch\"{u}tt and Werner
\cite{Nagy-Schutt-Werner-2019} for a comprehensive account and references. We start with the upper threshold. Note that
the $\Lambda^{\ast}$-condition is not required for this result.

\begin{theorem}\label{th:upper-mun}Let $\mu $ be an even probability measure on ${\mathbb R}$.
Then, for any $\delta\in \left(0,\tfrac{1}{2}\right)$ there exist $c(\mu,\delta)>0$ and $n_0(\mu,\delta )\in {\mathbb N}$ such that
$$\varrho_1(\mu_n ,\delta )\gr \left(1-\frac{c(\mu,\delta)}{\sqrt{n}}\right)\mathbb{E}_{\mu}(\Lambda_{\mu}^{*})$$
for all $n\gr n_0(\mu,\delta)$.
\end{theorem}

\begin{proof}The standard approach towards an upper threshold is based on the next fact which holds true in general, for
any Borel probability measure on ${\mathbb R}^n$. For every $r>0$ and every $N>n$ we have
\begin{equation}\label{eq:upper-1}{\mathbb E}_{\mu_n^N}(\mu_n (K_N))\ls \mu_n (B_r(\mu_n ))+ N\exp (-r).\end{equation}
This estimate appeared originally in \cite{DFM} and follows from the observation that
(by the definition of $\varphi_{\mu_n}$, Markov's inequality and the definition of $\Lambda_{\mu_n}^{\ast}$)
for every $x\in {\mathbb R}^n$ we have
\begin{equation}\label{eq:upper-2}\varphi_{\mu_n } (x)\ls\exp (-\Lambda_{\mu_n }^{\ast }(x)).
\end{equation}
We use \eqref{eq:upper-1} in the following way. Let $T_1:={\mathbb E}_{\mu}(\Lambda_{\mu}^{\ast})$ and
$T_n:=\mathbb{E}_{\mu_n}(\Lambda_{\mu_n}^{*})=T_1n$. Then, for all $\zeta\in (0,1)$,
from Chebyshev's inequality we have that
\begin{align*}
\mu_n(\{\Lambda_{\mu_n}^{*}\ls T_n -\zeta T_n  \})&\ls \mu_n(\{|\Lambda_{\mu_n}^{*}-T_n |\gr \zeta T_n \})
	\ls\frac{\mathbb{E}_{\mu_n}|\Lambda_{\mu_n}^{*}-T_n|^2}{\zeta^2T_n^2}=\frac{\beta(\mu_n)}{\zeta^2}=\frac{\beta(\mu)}{\zeta^2n}.
\end{align*}
Equivalently,
$$\mu_n (B_{(1-\zeta )T_n }(\mu_n))\ls \frac{\beta(\mu)}{\zeta^2n}.$$
Let $\delta\in\left(0,\tfrac{1}{2}\right)$. We may find $n_0(\mu,\delta)$ such that if $n\gr n_0(\mu,\delta)$ then $8\beta(\mu)/n<\delta <1/2$.
We choose $\zeta =\sqrt{2\beta(\mu)/n\delta }$, which implies that
$$\mu (B_{(1-\zeta )T_n}(\mu_n))\ls \frac{\delta }{2}.$$
From \eqref{eq:upper-1} we see that
\begin{align*}\sup\{ {\mathbb E}_{\mu_n^N}(\mu_n(K_N)):N\ls e^{(1-2\zeta )T_n} \}
&\ls \mu_n(B_{(1-\zeta)T_n}(\mu_n))+e^{(1-2\zeta )T_n }e^{-(1-\zeta)T_n}\\
&\ls \frac{\delta }{2}+e^{-\zeta T_n}\ls\delta ,
\end{align*}
provided that $\zeta T_n\gr \ln (2/\delta )$. Since $T_n=T_1n$, the last condition takes the form
$T_1n\gr c_1\ln (2/\delta )\sqrt{\delta n/\beta(\mu)}$ and it is certainly satisfied
if $n\gr n_0(\mu)$, where $n_0(\mu)$ depends only on $\beta(\mu)$ because $\sqrt{\delta}\ln(2/\delta )$ is
bounded on $\left(0,\tfrac{1}{2}\right)$. By the choice of $\zeta $ we conclude that
$$\varrho_1(\mu_n ,\delta )\gr \left(1-\sqrt{8\beta(\mu)/n\delta }\right)\mathbb{E}_{\mu}(\Lambda_{\mu}^{*})$$
as claimed.\end{proof}

For the proof of the lower threshold we need a basic fact that plays a main role in the proof of all
the lower thresholds that have been obtained so far. For a proof see \cite[Lemma~4.1]{Gatzouras-Giannopoulos-2009}.

\begin{lemma}\label{lem:inclusion}For every Borel subset $A$ of ${\mathbb R}^n$ we have that
$$1-\mu_n^N(K_N\supseteq A)\ls \binom{N}{n}p_{\mu}^{N-n}+2\binom{N}{n}\left (1-\inf_{x\in A}\varphi_{\mu_n }(x)\right)^{N-n}.$$
where $p_{\mu}=\max\{P(X=x):x\in {\mathbb R}\}<1$. Therefore,
\begin{equation}\label{eq:inclusion}{\mathbb E}_{\mu_n^N}\,[\mu_n (K_N)]\gr \mu_n (A)\left (1-\binom{N}{n}p_{\mu}^{N-n}-2\binom{N}{n}\left (1-\inf_{x\in A}\varphi_{\mu_n }(x)\right)^{N-n}\right).\end{equation}
\end{lemma}

We are going to apply Lemma~\ref{lem:inclusion} with  $A=B_{(1+\epsilon )T_n }(\mu_n)$, using Chebyshev's inequality exactly
as in the proof of Theorem~\ref{th:upper-mun}. From \eqref{eq:inclusion} it is clear that we will also need
a lower bound for $\inf_{x\in B_{(1+\epsilon )T_n}(\mu_n)}\varphi_{\mu_n }(x)$ which will imply that
$$2\binom{N}{n}\left (1-\inf_{x\in B_{(1+\epsilon )T_n}(\mu_n)}\varphi_{\mu_n }(x)\right)^{N-n}=o_n(1).$$
The main technical step is to obtain the next inequality.

\begin{theorem}\label{th:product}Let $\mu $ be an admissible even probability measure on ${\mathbb R}$
that satisfies the $\Lambda^{\ast}$-condition, i.e. $m(x)\sim\Lambda_{\mu}^{\ast }(x)$ as $x\uparrow x^{\ast}$.
Then, for every $\zeta >0,$ there exists $n_0(\mu,\zeta)\in\mathbb{N},$ depending only on $\zeta$
and $\mu$, such that for all $r>0$ and all $n\gr n_0(\mu,\zeta)$ we have that
$$\inf_{x\in B_r(\mu_n)}\varphi_{\mu_n }(x)\gr\exp (-(1+\zeta)r-2\zeta n).$$
\end{theorem}

\begin{proof}Let $x\in B_r(\mu_n)$ and $H_1$ be a closed half-space with $x\in\partial{H_1}$. There exists
$v\in {\mathbb R}^n\setminus\{0\}$ such that $H_1=\{y\in {\mathbb R}^n:\langle v,y-x\rangle\gr 0\}$. Consider
the function $q:B_r(\mu_n)\to {\mathbb R}$, $q(w)=\langle v,w\rangle$. Since $q$ is continuous and
$B_r(\mu_n)$ is compact, $q$ attains its maximum at some point $z\in B_r(\mu_n)$. Define
$H=\{y\in {\mathbb R}^n:\langle v,y-z\rangle\gr 0\}$. Then, $z\in\partial(H)$ and
for every $w\in B_r(\mu_n)$ we have $\langle v,w\rangle\ls \langle v,z\rangle $, which shows that
$\partial(H)$ supports $B_r(\mu_n)$ at $z$. Moreover, $H\subseteq H_1$ and hence
$P(\vec{X}\in H)\ls P(\vec{X}\in H_1)$. This shows that $\inf\{\varphi_{\mu_n}(x):x\in B_r(\mu_n)\}$ is attained
for some closed half-space $H$ whose bounding hyperplane supports $B_r(\mu_n )$.
Therefore, for the proof of the theorem it suffices to show that given $\zeta >0$ we may find
$n_0(\mu,\zeta )$ so that if $n\gr n_0(\mu,\zeta)$ then
\begin{equation}\label{eq:main}P\bigl(\vec{X}\in H\bigr)\gr \exp (-(1+\zeta)r-2\zeta n)\end{equation}
for any closed half-space $H$ whose bounding hyperplane supports
$B_r(\mu_n )$.

Let $H$ be such a half-space. Then, there exists $x\in\partial (B_r(\mu_n ))$ such that $$P\bigl(\vec{X}\in
H\bigr)=P\!\left(\sum\limits_{i=1}^n t_i(X_i-x_i)\gr
0\right),$$ where $t_i=h(x_i)$, because the normal vector to $H$ is $\nabla\Lambda_{\mu_n}^{\ast}(x)$
and $(\Lambda_{\mu}^{\ast})^{\prime}=h$ by Lemma~\ref{lem:2.5}\,(iii). We fix this $x$ for the rest of the proof.
By symmetry and independence we may assume that $x_i\gr 0$
for all $1\ls i\ls n$. Recall that $\Lambda_{\mu}^{\ast }(0)=0$ and that $\mu$ satisfies the $\Lambda^{\ast}$-condition:
we have $m(x)\sim\Lambda_{\mu}^{\ast }(x)$ as $x\uparrow x^{\ast}$. Therefore, we can find $M>\tau >0$ with
the following properties:
\begin{enumerate}
\item[(i)] If $0\ls x\ls \tau $ then $0\ls \Lambda_{\mu}^{\ast }(x)\ls\zeta$.
\item[(ii)] If $M<x<x^{\ast} $ then $P(X\gr x)\gr \exp(-\Lambda_{\mu}^{\ast }(x)(1+\zeta))$.
\end{enumerate}
Set $[n]=\{1,\ldots ,n\}$. We consider the sets of indices
\begin{align*}A_1=A_1(x)&:=\set{i\in [n]:x_i<\tau}\\
A_2=A_2(x)& :=\set{ i\in [n]:\tau \ls x_i\ls M},\\
A_3=A_3(x)&:=\set{i\in [n]:x_i>M}\end{align*}
and the probabilities
$$P_j=P_j(x):=P\!\left(\sum\limits_{i\in A_j}t_i(X_i-x_i)\gr 0\right)\qquad j=1,2,3.$$
By independence we have that
$$P\bigl(\vec{X}\in H\bigr)=P\!\left(\sum\limits_{i=1}^n t_i(X_i-x_i)\gr 0\right)\gr P_1P_2P_3.$$
We will give lower bounds for $P_1$, $P_2$ and $P_3$ separately.

\begin{lemma}\label{lem:product-6}We have that
$$P_1\gr \exp\!\left(-\sum_{i\in A_1}(\Lambda_{\mu}^{\ast }(x_i)+\zeta) -
c_1\ln\abs{A_1}-c_2\right),$$ where $c_1,c_2>0$
depend only on $\zeta $ and $\mu$.
\end{lemma}

\begin{proof}We write
\begin{equation}\label{eq:I1-delta}P_1=P\!\left(\sum\limits_{i\in A_1} t_i(X_i-x_i)\gr
0\right)\gr P\!\left (\sum_{i\in A_1}t_i(X_i-\tau )\gr 0\right
) ,\end{equation} and use the following fact (see \cite[Lemma~4.3]{Gatzouras-Giannopoulos-2009}):
For every $\tau\in (0,x^{\ast}),$ there exists $c(\tau)>0$
depending only on $\tau$ and $\mu,$ such that for any
$k\in\mathbb{N}$ and any $v_1,\ldots,v_k\in\mathbb{R}$ with
$\sum_{i=1}^k v_i>0$ we have that $$P\! \left ( \sum_{i=1}^k
v_i(X_i-\tau)\gr  0\right ) \gr c(\tau)\, k^{-3/2}\, e^{-k
\Lambda_{\mu}^{\ast }(\tau)} .$$
Combining the above with \eqref{eq:I1-delta} and using the simple bound
$\Lambda_{\mu}^{\ast }(\tau )\ls\zeta \ls \Lambda_{\mu}^{\ast}(x)+\zeta $ for $x$ in $[0,\tau]$,
we conclude the proof of the lemma.\end{proof}

\begin{lemma}\label{lem:product-7}We have that
$$P_3\gr \exp\!\left(-(1+\zeta)\sum_{i\in
A_3}\Lambda_{\mu}^{\ast }(x_i)\right ).$$
\end{lemma}

\begin{proof}By independence, we can write
$$P_3=P\!\left(\sum\limits_{i\in A_3} t_i(X_i-x_i)\gr 0\right)\gr \prod_{i\in A_3}P(X_i\gr x_i).$$
By the choice of $M$ we see that
$$P(X_i\gr x_i)\gr e^{-\Lambda_{\mu}^{\ast }(x_i)(1+\zeta )}$$
for all $i\in A_3$, and this immediately gives the lemma.\end{proof}

\begin{lemma}\label{lem:product-10} There exist $c_3,c_4>0$ depending only on $\zeta, M$ and $\mu,$ such that
$$P\!\left(\sum\limits_{i\in A_2} t_i(X_i-x_i)\gr 0\right )\gr\exp\!\left(-\sum_{i\in A_2}
\Lambda_{\mu}^{\ast }(x_i)-c_3\sqrt{|A_2|}-c_4\right).$$
\end{lemma}

The proof of this estimate requires some preparation. Without loss
of generality, we may assume that $A_2=\set{1,\ldots ,k}$ for some
$k\ls n$. Recall that $t_i= h(x_i)$ for each $i$, and that this is
equivalent to having $x_i=\Lambda_{\mu}^{\prime }(t_i)$ for each $i$ (see Lemma~\ref{lem:2.5}\,(ii){}). Define
the probability measure $P_{x_1,\ldots,x_k}$ on
$(\varOmega ,\mathcal{F})$, by
$$P_{x_1,\ldots,x_k}(A) := {\mathbb E}\! \left[ \mathds{1}_A \cdot
\exp\!\left(\sum\limits_{i=1}^k
(t_iX_i-\Lambda_{\mu}(t_i))\right)\right]$$ for
$A\in\mathcal{F}$. Direct computation shows that, under $P_{x_1,\ldots,x_k},$
the random variables $t_1X_1,\ldots,t_kX_k$ are
independent\textup{,} with mean\textup{,} variance and absolute
central third moment given by
\begin{align*} {\mathbb E}_{x_1,\ldots,x_k}(t_iX_i) &=
t_i\Lambda_{\mu}^{\prime }(t_i)=t_ix_i,\\ {\mathbb E}_{x_1,\ldots,x_k}\bigl(
\abs{t_i(X_i-x_i)}^2\bigr) &= t_i^2\Lambda_{\mu}''(t_i) ,\\
{\mathbb E}_{x_1,\ldots,x_k} \bigl(\abs{t_i(X_i-x_i)}^3\bigr) &= \abs{t_i}^3
{\mathbb E}_{t_i}\bigl(\abs{X-\Lambda_{\mu}'(t_i)}^3\bigr),\end{align*} respectively.
Set $\sigma_i^2:=t_i^2\Lambda_{\mu}''(t_i)$, $$s_k^2 := \sum\limits_{i=1}^k
{\mathbb E}_{x_1,\ldots,x_k}\bigl(\abs{t_i(X_i-x_i)}^2\bigr) =
\sum\limits_{i=1}^k t_i^2\Lambda_{\mu}^{\prime\prime }(t_i) =
\sum\limits_{i=1}^k \sigma_i^2$$ and $$S_k
:=\sum\limits_{i=1}^k t_i(X_i-x_i) ,$$ and let
$F_k\colon\mathbb{R}\rightarrow\mathbb{R}$ denote the cumulative
distribution function of the random variable $S_k/s_k$ under the
probability law $P_{x_1,\ldots,x_k}$: $F_k(x) :=
P_{x_1,\ldots,x_k}(S_k\ls x s_k)$ $(x\in\mathbb{R})$. Write also
$\nu_k$ for the probability measure on $\mathbb{R}$ defined by
$\nu_k(-\infty,x] := F_k(x)$ $(x\in\mathbb{R})$. Notice that
${\mathbb E}_{x_1,\ldots,x_k}(S_k/s_k)=0$ and $
\operatorname{Var}_{x_1,\ldots,x_k} (S_k/s_k)=1$.

\begin{lemma}\label{lem:product-8}
The following identity holds\textup{:}
$$P\!\left(\sum\limits_{i=1}^k t_i(X_i-x_i)\gr 0\right) =
\left(\int_{[0,\infty)} e^{-s_ku}\, d\nu_k(u)\right)\, \exp\!\left(-
\sum_{i=1}^k \Lambda_{\mu}^{\ast }(x_i)\right) .$$
\end{lemma}

\begin{proof}By definition of the measure
$P_{x_1,\ldots,x_k}$, we have that
\begin{equation*}P\!\left(\sum\limits_{i=1}^k t_i(X_i-x_i)\gr
0\right)=P(S_k\gr 0)
= {\mathbb E}_{x_1,\ldots,x_k}\! \left[
\mathds{1}_{[0,\infty)}(S_k)\cdot \exp\!\left(-\sum_{i=1}^k (t_i
X_i-\Lambda_{\mu}(t_i))\right)\right] .\end{equation*} It follows that
$$P\!\left(\sum\limits_{i=1}^k t_i(X_i-x_i)\gr  0\right) =
\int_{[0,\infty)} e^{-s_ku}\, d\nu_k(u)\cdot \exp\!\left(\sum_{i=1}^k
(\Lambda_{\mu}(t_i) - t_ix_i)\right) ,$$
and the lemma now follows from Lemma~~\ref{lem:2.5}\,(ii). \end{proof}

\medskip

We will also use the following consequence of the Berry-Esseen
theorem (cf.\ \cite{Fe2}, p.\ 544).

\begin{lemma}\label{lem:product-9}For any $a,b>0,$ there exist $k_0\in {\mathbb N}$ and
$\theta >0$ with the following property\textup{:} If $k\gr k_0,$ and
if $Y_1,\ldots ,Y_k$ are independent random variables with
$$\mathbb{E}(Y_i)=0,\quad \sigma_i^2:=\mathbb{E}(Y_i^2)\gr a,\quad
\mathbb{E}(|Y_i|^3)\ls b,$$ then $$\mathbb{P}\left
(0\ls\sum_{i=1}^k Y_i\ls\sigma\right )\gr\theta ,$$
where $\sigma^2=\sigma_1^2+\cdots +\sigma_k^2$.
\end{lemma}

\begin{proof}[Proof of Lemma~$\ref{lem:product-10}$]Consider the random variables $Y_i:=t_i(X_i-x_i)$, $i\in A_2=\{1,\ldots ,k\}$, which are
independent with respect to $P_{x_1,\ldots ,x_k}$ and satisfy ${\mathbb E}_{x_1,\ldots ,x_k}(Y_i)=0$
for all $1\ls i\ls k$. Set $J_{\mu}^{\ast}=(\Lambda_{\mu}^{\prime})^{-1}([\tau , M])$. Since $\tau\ls x_i\ls M$
for all $1\ls i\ls k$, we see that
$$\sigma_i^2={\mathbb E}_{x_1,\ldots,x_k}\,(Y_i^2) = t_i^2\Lambda_{\mu}''(t_i)\gr
\min_{t\in J_{\mu}^{\ast}}t^2\Lambda_{\mu}''(t)=:a_1>0 $$ and
$${\mathbb E}_{x_1,\ldots,x_k}\,(|Y_i|^3)= \abs{t_i}^3{\mathbb E}_{t_i}\bigl(\abs{X-\Lambda_{\mu}'(t_i)}^3\bigr)\ls
\max_{t\in J_{\mu}^{\ast}}\abs{t}^3{\mathbb E}_t\bigl(\abs{X-\Lambda_{\mu}'(t)}^3\bigr)=:b_1 <+\infty $$ for
all $1\ls i\ls k$. Applying Lemma~\ref{lem:product-9} we find $\theta>0$ and $k_0\in {\mathbb N}$ such that
if $k\gr k_0$ then
\begin{equation}\label{eq:zeta}\mathbb{P}_{x_1,\ldots ,x_k}\left (0\ls\sum_{i=1}^k Y_i\ls s_k\right )\gr\theta .\end{equation}
Now, we distinguish two cases:

\smallskip

\noindent \textit{Case 1}: Assume that $k<k_0$. Then, working as for $A_3$, we see that
$$P\!\left(\sum\limits_{i\in A_2} t_i(X_i-x_i)\gr  0\right)\gr
\prod\limits_{i\in A_2} P(X_i\gr x_i) \gr \prod\limits_{i\in A_2} P(X_i\gr M)= e^{-m(M)k}\gr e^{-m(M)k_0}.$$

\noindent \textit{Case 2}: Assume that $k\gr k_0$. From Lemma~\ref{lem:product-8} we have
\begin{align}\label{eq:4.30}P\!\left(\sum\limits_{i\in A_2} t_i(X_i-x_i)\gr 0\right)
&=\left(\int_{[0,\infty)} e^{-s_ku}\, d\nu_k(u)\right)\, \exp\!\left(-
\sum_{i=1}^k \Lambda_{\mu}^{\ast }(x_i)\right)\\
\nonumber &\gr e^{-s_k}\nu_k([0,1]) \exp\!\left(-\sum_{i\in A_2}\Lambda_{\mu}^{\ast }(x_i)\right).\end{align}
From \eqref{eq:zeta} we see that
$$\nu_k([0,1])=P_{x_1,\ldots ,x_k}(0\ls S_k\ls s_k)=\mathbb{P}\left (0\ls\sum_{i=1}^k Y_i\ls s_k\right )\gr\theta .$$
Moreover, $s_k\ls c\sqrt{k}$.
Combining the two cases we get the estimate of Lemma~\ref{lem:product-10} for $P_2$.\end{proof}

We can now complete the proof of Theorem~\ref{th:product}. Collecting the estimates from
Lemma~\ref{lem:product-6}, Lemma~\ref{lem:product-7} and Lemma~\ref{lem:product-10}, we may write
\begin{align*}
P\!\left(\sum\limits_{i=1}^n t_i(X_i-x_i)\gr 0\right ) &\gr
\prod_{j=1}^3 P\! \left(\sum\limits_{i\in A_j} t_i(X_i-x_i)\gr
0\right) \\ &\gr \exp\!\left(-\sum_{i=1}^n \Lambda_{\mu}^{\ast }(x_i)\right)\\
&\quad\times \exp\!\left(-\zeta \abs{A_1} -c_1\ln\abs{A_1}-c_2
-\zeta \sum_{i\in A_3} \Lambda_{\mu}^{\ast }(x_i)-c_3\sqrt{\abs{A_2}} -
c_4\right)\\ &\gr \exp\! \left(-\sum_{i=1}^n \Lambda_{\mu}^{\ast }(x_i)-
\zeta \sum_{i=1}^n \Lambda_{\mu}^{\ast }(x_i) -2\zeta n
\right),\end{align*} provided $n\gr n(\mu,\zeta)$ for an
appropriate $n(\mu,\zeta)\in\mathbb{N}$ depending only on
$\zeta$ and $\mu$. This proves \eqref{eq:main}.
\end{proof}

We are now able to provide an upper bound for $\varrho_2(\mu_n ,\delta )$.

\begin{theorem}\label{th:lower-mun}Let $\mu $ be an admissible even probability measure on ${\mathbb R}$
that satisfies the $\Lambda^{\ast}$-condition, i.e. $m(x)\sim\Lambda_{\mu}^{\ast }(x)$ as $x\uparrow x^{\ast}$.
Then, for any $\delta\in \left(0,\tfrac{1}{2}\right)$ and $\epsilon\in (0,1)$ we can find $n_0(\mu,\delta,\epsilon )$ such that
$$\varrho_2(\mu_n ,\delta )\ls \left(1+\epsilon \right)\mathbb{E}_{\mu}(\Lambda_{\mu}^{*})$$
for all $n\gr n_0(\mu,\delta,\epsilon )$.
\end{theorem}

\begin{proof}Let $\epsilon\in (0,1)$ and define $\zeta=T_1\epsilon /(3T_1+4)$. Note that if $T_n:=\mathbb{E}_{\mu_n}(\Lambda_{\mu_n}^{*})$ then,
as in the proof of Theorem~\ref{th:upper-mun}, Chebyshev's inequality implies that
\begin{equation*}
\mu_n(\{\Lambda_{\mu_n}^{*}\gr T_n +\zeta T_n  \})\ls \mu_n(\{|\Lambda_{\mu_n}^{*}-T_n |\gr \zeta T_n \})\ls\frac{\beta(\mu)}{\zeta^2n}.
\end{equation*}
Since $\zeta $ depends only on $\epsilon$ and $\mu$  we can find $n_0(\mu,\delta,\epsilon)$ such that
$$\frac{\beta(\mu)}{\zeta^2n}\ls\frac{\delta}{2}$$
and hence
$$\mu_n (B_{(1+\zeta )T_n }(\mu_n))\gr 1-\frac{\delta}{2}$$
for all $n\gr n_0(\mu,\delta,\epsilon)$. Assume that $N\gr\exp ((1+\epsilon)T_n)=\exp((1+3\zeta)T_n+4\zeta n)$.
Applying \eqref{eq:inclusion} with $A=B_{(1+\zeta )T_n }(\mu_n)$
and using the estimate of Theorem~\ref{th:product} we get
\begin{equation}\label{eq:inclusion-2}{\mathbb E}_{\mu_n^N}\,[\mu_n (K_N)]\gr \mu_n (B_{(1+\zeta )T_n }(\mu_n))\left (1-\binom{N}{n}p_{\mu}^{N-n}-2\binom{N}{n}\left (1-\exp(-(1+\zeta)^2T_n-2\zeta n)\right)^{N-n}\right).\end{equation}
Therefore, taking into account the fact that $(1+\zeta)^2<1+3\zeta$ for $\zeta<1$, we will have that
$$\varrho_2(\mu_n,\delta )\ls (1+\epsilon)T_1$$
if we check that
$$\binom{N}{n}\, p_{\mu}^{N-n}+2\binom{N}{n}\left (1-\exp(-(1+3\zeta)T_1n-2\zeta n)\right)^{N-n}\ls\frac{\delta}{2}.$$
We first claim that there exists $n_1(\mu,\delta)$ such that
$$\binom{N}{n}p_{\mu}^{N-n}<\frac{\delta}{4}$$
for all $n\gr n_1(\mu,\delta)$.  Indeed, since $\binom{N}{n}\ls (eN/n)^n$, it suffices to check that
\begin{equation}\label{eq:p}1+\ln\! \left(\frac{N}{n}\right) + \frac{N-n}{n}\, \ln
p_{\mu} < \frac{1}{n}\ln (\delta /4) .\end{equation}  Set $x:=N/n$. Then, \eqref{eq:p} is
equivalent to $$(x-1)\ln (1/p_{\mu}) -\ln x
>1+\frac{1}{n}\ln (4/\delta).$$ The claim follows from the facts that the function on
the left-hand side increases to infinity as $x\to\infty$, and
$x=N/n\gr\exp((1+3\zeta)T_1n+4\zeta n)/n\gr e^{4\zeta n}/n \to\infty $ when
$n\to\infty $.

Next we check that there exists $n_2(\mu,\delta, \epsilon )$ such that
$$2\,\binom{N}{n}\bigl[1-\exp(-(1+3\zeta)T_1n-2\zeta n)\bigr]^{N-n}<\frac{\delta}{4}$$
for all $n\gr n_2(\mu,\delta, \epsilon )$. Since $1-y\ls e^{-y}$, it suffices to check that
\begin{equation}\label{eq:41}\left(\frac{2eN}{n}\right)^n\exp\! \left(-(N-n)
\exp(-(1+3\zeta)T_1n-2\zeta n)\right)<\frac{\delta}{4}\end{equation} for all $n\gr n_2$.
Setting $x:=N/n$, we see that this inequality is equivalent to
$$\exp((1+3\zeta)T_1n+2\zeta n)<\frac{x-1}{\ln x+\ln(2e)+n^{-1}\ln (4/\delta)}.$$
Since $N\gr \exp((1+3\zeta)T_1n+4\zeta n)$, we easily check that
the right-hand side exceeds $\exp((1+3\zeta)T_1n+3\zeta n)$ when $n\gr n_2(\mu,\zeta,\delta )=n_2(\mu,\epsilon,\delta)$,
and hence we get \eqref{eq:41}. Combining the above we conclude that
$$\varrho_2(\mu_n,\delta)\ls \left(1+\epsilon \right)T_1$$
for all $n\gr n_0$, where $n_0= n_0(\mu,\delta,\epsilon )$ depends only on $\mu$, $\delta $ and $\epsilon $.
\end{proof}

\begin{proof}[Proof of Theorem~$\ref{th:final}$]Let $\delta\in\left(0,\tfrac{1}{2}\right)$ and $\varepsilon \in (0,1)$.
From the estimates of Theorem~\ref{th:upper-mun} and Theorem~\ref{th:lower-mun} we see that there exists
$n_0(\mu ,\delta ,\varepsilon)$ such that if $n\gr n_0$ then $\frac{c(\mu,\delta)}{\sqrt{n}}<\epsilon $ (where
$c(\mu ,\delta)$ is the constant in Theorem~\ref{th:upper-mun}) and
$$\varrho_1(\mu_n ,\delta )\gr \left(1-\frac{c(\mu,\delta)}{\sqrt{n}}\right)\mathbb{E}_{\mu}(\Lambda_{\mu}^{*})$$
as well as
$$\varrho_2(\mu_n ,\delta )\ls \left(1+\epsilon \right)\mathbb{E}_{\mu}(\Lambda_{\mu}^{*}).$$
Therefore,
$$\varrho (\mu_n,\delta)\ls 2\varepsilon \mathbb{E}_{\mu}(\Lambda_{\mu}^{*})$$
for all $n\gr n_0$. Since $\epsilon\in (0,1)$ was arbitrary, we see that $\lim\limits_{n\to\infty}\varrho(\mu_n,\delta)\to 0$,
as claimed in Theorem~\ref{th:final}.\end{proof}

\section{Threshold for the $p$-measures}\label{section:p}

We write $\nu$ for the symmetric exponential distribution on ${\mathbb R}$; thus, $\nu$ is the probability
measure with density $\frac{1}{2}\exp(-|x|)$. More generally, for any $p\gr 1$ we denote by $\nu_p$ the probability
distribution on ${\mathbb R}$ with density $(2\gamma_p)^{-1}\exp(-|x|^p)$, where $\gamma_p=\Gamma(1+1/p)$.
Note that $\nu_1=\nu$. The product measure $\nu_p^n=\nu_p^{\otimes n}$ has density $(2\gamma_p)^{-n}\exp(-\|x\|_p^p)$,
where $\|x\|_p=\left(\sum_{i=1}^n|x_i|^p\right)^{1/p}$ is the $\ell_p^n$-norm.

Our aim in this section is to show that $\nu_p$ satisfies the $\Lambda^{\ast}$-condition.

\begin{theorem}\label{th:p-condition}For any $p\gr 1$ we have that $-\ln (\nu_p[x,\infty))\sim \Lambda_{\nu_p}^{\ast}(x)$ as $x\to\infty$.
In other words,
\begin{equation}\label{eq:p-condition}\lim_{x\to +\infty}\frac{-\ln (\nu_p[x,\infty))}{\Lambda_{\nu_p}^{\ast}(x)}=1.\end{equation}
\end{theorem}

\begin{proof}[Proof of the case $p=1$]We start with the case $p=1$ which is simple because $\Lambda_{\nu}^{\ast}$ can be computed explicitly.
A direct calculation shows that $$\Lambda_{\nu}^{\ast}(x)=\sqrt{1+x^2}-1-\ln\left(\frac{\sqrt{1+x^2}+1}{2}\right),\qquad x\in {\mathbb R}.$$
It follows that $\Lambda_{\nu}^{\ast}(x)\sim x$ as $x\to \infty$. On the other hand,
$\nu([x,\infty))=\frac{1}{2}e^{-x}$ for all $x>0$, which shows that $-\ln (\nu([x,\infty))=x+\ln 2$, and hence
$-\ln (\nu [x,\infty))\sim x$ as $x\to\infty$. Combining the above we immediately see that \eqref{eq:p-condition} is satisfied for $p=1$.\end{proof}

For the rest of this section we fix $p>1$. Following \cite{BGT-book} we say that a non-negative function $f:{\mathbb R}\to {\mathbb R}$ is regularly varying
of index $s\in {\mathbb R}$, and write $f\in R_s$, if $\lim\limits_{x\to\infty}f(\lambda x)/f(x)=\lambda^s$ for every $\lambda >0$. It is proved
in \cite[Theorem~4.12.10]{BGT-book} that if $f\in R_s$ for some $s>0$ then
$$-\ln \left(\int_x^{\infty}e^{-f(t)}dt\right)\sim f(x)$$
as $x\to\infty$. Let $f_p(x)=|x|^p$, $x\gr 0$. It is clear that $f_p\in R_p$, and hence
$$-\ln (\nu_p[x,\infty))=-\ln \left((2\gamma_p)^{-1}\int_x^{\infty}e^{-f_p(t)}dt\right)
=\ln (2\gamma_p)-\ln \left(\int_x^{\infty}e^{-f_p(t)}dt\right)\sim f_p(x)$$
as $x\to\infty $. This proves the following.

\begin{lemma}\label{lem:p-condition-1}For every $p\gr 1$ we have that $-\ln (\nu_p[x,\infty))\sim x^p$ as $x\to\infty $.
\end{lemma}

Lemma~\ref{lem:p-condition-1} shows that in order to complete the proof of the theorem we have to show
that $\Lambda_{\nu_p}^{\ast}(x)\sim x^p$ as $x\to\infty $. Let $g_p(x)=x^2$ for $0\ls x< 1$ and $g_p(x)=x^p$ for $x\gr 1$. It is shown in \cite{Latala-Wojtaszczyk-2008} that for any $p\gr 1$ and $x\in {\mathbb R}$ one has
$$\Lambda_{\nu_p}^{\ast}(x/c)\ls g_p(|x|)\ls \Lambda_{\nu_p}^{\ast}(cx)$$
where $c>1$ is an absolute constant. 

For the proof of $\Lambda_{\nu_p}^{\ast}(x)\sim x^p$ as $x\to\infty $ we shall apply the Laplace method; more precisely, we
shall use the next version of Watson's lemma (see equation (2.34) in \cite[Section~2.2]{Murray-book}).

\begin{proposition}\label{prop:murray}Let $S<a<T\ls\infty $ and $g,h:[S,T]\to {\mathbb R}$, where
$g$ is continuous with a Taylor series in a neighborhood of $a$,
and $h$ is twice continuously differentiable and has its maximum at $a$ and satisfies $h^{\prime}(a)=0$
and $h^{\prime\prime}(a)<0$. Assume also that the integral
$$\int_S^Tg(x)e^{th(x)}\,dx$$
converges for large values of $t$. Then,
$$\int_S^Tg(x)e^{th(x)}\,dx\sim g(a)\left(-\frac{2\pi}{th^{\prime\prime}(a)}\right)^{1/2}e^{th(a)}+e^{th(a)}O(t^{-3/2})$$
as $t\to +\infty$.
\end{proposition}

We apply Proposition~\ref{prop:murray} to get the next asymptotic estimate.

\begin{lemma}\label{lem:murray-1}Let $p>1$ and $q$ be the conjugate exponent of $p$. Then, setting $y=t^q$ we have that
$$I(t):=\int_0^{\infty}e^{tx-x^p}\,dx\sim y^{\frac{1}{p}}e^{yh(a)}\left[\left(-\frac{2\pi}{yh^{\prime\prime}(a)}\right)^{1/2}+O(y^{-3/2})\right]$$
as $t\to +\infty$, where $h(s)=s-s^p$ on $[0,\infty)$ and $a=p^{-q/p}$.
\end{lemma}

\begin{proof}We set $x=\lambda s$ and $t=\lambda^{p-1}$. Then,
$$I(t)=I(\lambda^{p-1})=\lambda \int_0^{\infty}e^{\lambda^p(s-s^p)}\,ds.$$
Now, set $y=\lambda^p=t^q$. Then,
$$I(t)=y^{1/p}\int_0^{\infty}e^{y(s-s^p)}ds.$$
We have $h^{\prime}(s)=1-ps^{p-1}$, therefore $h$ attains its maximum at $a=(1/p)^{\frac{1}{p-1}}=p^{-q/p}$.
Now, applying Proposition~\ref{prop:murray} with $g\equiv 1$ we see that
$$\int_0^{\infty }e^{yh(s)}ds\sim e^{yh(a)}\left[\left(-\frac{2\pi}{yh^{\prime\prime}(a)}\right)^{1/2}+O(y^{-3/2})\right],$$
and the lemma follows.\end{proof}

We proceed to study the asymptotic behavior of $\Lambda_{\nu_p}(t)$. Recall that
$$\Lambda_{\nu_p}(t)=\ln\left(c_p\int_{-\infty}^{\infty}e^{tx-|x|^p}\,dx\right),$$
where $c_p=(2\Gamma (1+1/p))^{-1}$. By the dominated convergence theorem,
$$\int_{-\infty}^0e^{tx-|x|^p}\,dx\longrightarrow 0$$
as $t\to +\infty $. Therefore, from Lemma~\ref{lem:murray-1},
$$c_p\int_{-\infty}^{\infty}e^{tx-|x|^p}\,dx\sim c_p\int_0^{\infty}e^{tx-x^p}\,dx
\sim c_py^{\frac{1}{p}}e^{yh(a)}\left[\left(-\frac{2\pi}{yh^{\prime\prime}(a)}\right)^{1/2}+O(y^{-3/2})\right],$$
where $h(s)=s-s^p$ on $[0,\infty)$, $a=p^{-q/p}$ and $y=t^q$. Now,
\begin{equation*}
\ln\left(c_py^{\frac{1}{p}}e^{yh(a)}\left[\left(-\frac{2\pi}{yh^{\prime\prime}(a)}\right)^{1/2}+O(y^{-3/2})\right]\right)
=\ln c_p+\frac{1}{p}\ln y+yh(a)+O(\ln y)\sim yh(a).
\end{equation*}
It follows that $\Lambda_{\nu_p}(t)\sim yh(a)= (p^{-q/p}-p^{-q})t^q$, where $q$ is the conjugate exponent of $p$.
We rewrite this as follows.

\begin{lemma}\label{lem:murray-2}Let $p>1$ and $q$ be the conjugate exponent of $p$. Then,
$$\Lambda_{\nu_p}(t)\sim \frac{p-1}{p^q}t^q\quad\hbox{as}\;\;t\to +\infty .$$
\end{lemma}

\noindent Lemma~\ref{lem:murray-2} allows us to determine the asymptotic behavior of $\Lambda_{\nu_p}^{\ast}(x)$ as
$x\to\infty$. We need a lemma which appears in \cite{de-Bruijn-book} and \cite{Olver-book}.

\begin{lemma}\label{lem:de-bruijn}Let $q\gr 1$, $a>0$ and $f:[a,\infty )\to {\mathbb R}$ be a continuously differentiable
function such that $f^{\prime}$ is increasing on $[a,\infty)$ and $f(t)\sim t^q$ as $t\to +\infty$. Then,
$f^{\prime}(t)\sim qt^{q-1}$ as $t\to +\infty$.
\end{lemma}

\begin{proof}[Sketch of the proof]Let $\varepsilon\in (0,1)$. There exists $b>a$ and $\eta :[b,\infty )\to {\mathbb R}$ such that
$|\eta(t)|\ls\varepsilon $ and $f(t)=t^q(1+\eta(t))$ for all $t>b$. Since $f^{\prime}$ is increasing, for any $s>0$ we have that
\begin{align*}
sf^{\prime}(t) &\ls \int_t^{t+s}f^{\prime}(u)\,du=f(t+s)-f(t)=\big((t+s)^q-t^q\big)+\big((t+s)^q\eta(t+s)-t^q\eta(t)\big)\\
&\ls sq(t+s)^{q-1}+2\varepsilon (t+s)^q.
\end{align*}
We set $s=\sqrt{\varepsilon}t$. Then,
$$f^{\prime}(t)\ls qt^{q-1}\big((1+\sqrt{\varepsilon})^{q-1}+2q^{-1}\sqrt{\varepsilon}(1+\sqrt{\varepsilon})^q\big)$$
for all $t>b$. In the same way we see that
$$f^{\prime}(t)\gr qt^{q-1}\big((1-\sqrt{\varepsilon})^{q-1}-2q^{-1}\sqrt{\varepsilon}\big)$$
for all $t>b/(1-\sqrt{\varepsilon})$, and the lemma follows.\end{proof}

We also need the next simple lemma.

\begin{lemma}\label{lem:minas}Let $a>0$ and $f:[a,+\infty)\to {\mathbb R}$ be a strictly increasing function.
Assume that for some $C>0$ and $p>1$ we have $f(x)\sim Cx^p$ as $x\to +\infty$, and that
$\lim\limits_{y\to +\infty}f^{-1}(y)=+\infty $. Then, $f^{-1}(y)\sim (y/C)^{1/p}$ as $y\to +\infty $.
\end{lemma}

\begin{proof}We may write $f(x)=Cx^pg(x)$ for some function $g:[a,+\infty)\to {\mathbb R}$ with
$\lim\limits_{x\to +\infty}g(x)=1$. Then, for sufficiently large $x$ we have that
$x=\left(\frac{f(x)}{C}\cdot\frac{1}{g(x)}\right)^{1/p}$. It follows that, for sufficiently large $y$,
$$f^{-1}(y)=\left(\frac{y}{C}\frac{1}{g(f^{-1}(y))}\right)^{1/p},$$
and the lemma follows because $\lim\limits_{y\to +\infty}f^{-1}(y)=+\infty$ and $\lim\limits_{x\to +\infty}g(x)=1$.
\end{proof}

\smallskip 

\begin{proof}[Proof of the case $p>1$ in Theorem~$\ref{th:p-condition}$]Now, we can show that
\begin{equation}\label{eq:asymptotic-final}\Lambda_{\nu_p}^{\ast}(x)\sim x^p\end{equation}
as $x\to\infty$. We know that $\Lambda_{\nu_p}^{\ast}(x)=xh(x)-\Lambda_{\nu_p} (h(x))$ where $h(x)=(\Lambda_{\nu_p}^{\prime})^{-1}(x)$.
From Lemma~\ref{lem:murray-2} and Lemma~\ref{lem:de-bruijn} we see that $\Lambda_{\nu_p}^{\prime}(t)\sim p^{-(q-1)}t^{q-1}$,
and Lemma~\ref{lem:minas} implies that
$$h(x)\sim px^{\frac{1}{q-1}}=px^{p-1},$$
using also the fact that $(p-1)(q-1)=1$. It follows that
$$\frac{\Lambda_{\nu_p}^{\ast}(x)}{x^p} = \frac{h(x)}{x^{p-1}}-\frac{\Lambda_{\nu_p} (h(x))}{x^p} =
\frac{h(x)}{x^{p-1}}-\frac{\Lambda_{\nu_p} (h(x))}{h(x)^{\frac{p}{p-1}}}\left(\frac{h(x)^{\frac{1}{p-1}}}{x}\right)^p
\longrightarrow p-\frac{p-1}{p^q}\cdot p^q=1$$
as $x\to\infty $. This proves \eqref{eq:asymptotic-final} and completes the proof of the theorem.\end{proof}

\bigskip

\noindent {\bf Acknowledgement.} The author is grateful to the referee for very useful comments and suggestions on the
presentation of the results of this article. He acknowledges support by the Hellenic Foundation for
Research and Innovation (H.F.R.I.) under the ``Third Call for H.F.R.I. PhD Fellowships" (Fellowship Number: 5779).

\bigskip


\footnotesize
\bibliographystyle{amsplain}

\small

\medskip

\thanks{\noindent {\bf Keywords:} threshold, random polytopes, convex bodies, half-space depth, Cramer transform.

\smallskip

\thanks{\noindent {\bf 2020 MSC:} Primary 60D05; Secondary 60E15, 62H05, 52A22, 52A23.}

\bigskip

\bigskip

\noindent \textsc{Minas~Pafis}: Department of
Mathematics, National and Kapodistrian University of Athens, Panepistimioupolis 157-84,
Athens, Greece.

\smallskip

\noindent \textit{E-mail:} mipafis@math.uoa.gr

\bigskip

\end{document}